\documentclass[sn-mathphys]{sn-jnl}

\jyear{2022}%

\theoremstyle{thmstyleone}%
\newtheorem{theorem}{Theorem}
\newtheorem{corollary}[theorem]{Corollary}

\newtheorem{proposition}[theorem]{Proposition}%

\theoremstyle{thmstyletwo}%
\newtheorem{example}{Example}%

\theoremstyle{thmstylethree}%
\usepackage{amssymb}
\usepackage{amsmath}
\usepackage{mathtools}

\newcommand{\ve}{\varepsilon}
\newcommand{\vk}{\varkappa}
\renewcommand{\R}{\mathbb{R}}
\newcommand{\Y}{\mathcal{Y}}

\begin{document}

\title[Long-Time Influence of Perturbations]{Perturbation of Systems with a First Integral: Motion on the Reeb Graph}


\author*[1]{\fnm{Mark I.} \sur{Freidlin}}\email{mif@umd.edu}
\affil*[1]{\orgdiv{Department of Mathematics}, \orgname{University of Maryland}, \orgaddress{\city{College Park}, \postcode{20742}, \state{MD}, \country{USA}}}

\maketitle

\section*{Abstract}
	We consider the long-time behavior of systems close to a system with a smooth first integral. Under certain assumptions, the limiting behavior, to some extent, turns out to be universal: it is determined by the first integral, the deterministic perturbation, and the initial point. Furthermore, it is the same for a broad class of noises. In particular, the long-time behavior of a deterministic system can, in a sense, be stochastic and stochastic systems can have a reduced stochasticity. The limiting distribution is calculated explicitly.

\section{Introduction}
Let $f(x), x \in \R^d,$ be a regular enough vector field and $H(x)$ be a smooth first integral of the system 
\begin{equation} \label{eq.smooth_first_int}
	\dot{X}(t) = f(X(t)), \quad X(0) = x \in \R^d
\end{equation}
This means that $H(X(t)) \equiv H(x)$, what is equivalent to the identity $\nabla H(x) \cdot f(x) \equiv 0$. Consider a small perturbation of (\ref{eq.smooth_first_int}):
\begin{equation} \label{eq.smooth_pert}
	\dot{\tilde{X}}^\ve(t) = f(\tilde{X}^\ve(t)) + \ve b(\tilde{X}^\ve(t)), \quad \tilde{X}^\ve(0) = x,
\end{equation}
where $b(x)$ is a bounded smooth vector field with $0 < \ve \ll 1$.

\noindent An important class of dynamical system having a first integral is given by the Hamiltonian systems:
\begin{quote}
	Let $d=2n$, $x = (p_1, ..., p_n, q_1, ..., q_n) \in \R^{2n}$, $f(x) = \overline{\nabla}H(x) = (- \nabla_q H(p,q), \nabla_p H(p,q))$, where $\nabla_q H$ and $\nabla_p H$ are the gradients of $H(p,q)$ in $q$ and $p$ respectively. Then the system (\ref{eq.smooth_first_int}) has the form
	\begin{equation}
		\dot{X}(t) = \overline{\nabla} H(X(t)), \quad X(0) = x \in \R^{2n} \label{eq.e3}
	\end{equation}
	and the perturbed system is 
	\begin{equation}
	\dot{\tilde{X}}^\ve(t) = \overline{\nabla}H(\tilde{X}^\ve(t)) + \ve b(\tilde{X}^\ve(t)), \quad \tilde{X}^\ve(0) = x. \label{eq.e4}
	\end{equation}
\end{quote}
The Hamiltonian $H(x)$ is a first integral of system (\ref{eq.e3}). Of course, $\tilde{X}^\ve(t)$ will be uniformly close to $X(t)$ on each finite time interval $[0,T]$ as $\ve$ is close to zero. But $\tilde{X}^\ve(t)$ can deviate from $X(t)$ on a distance of order 1 as $\ve \rightarrow 0$ on time intervals of order $\ve^{-1}$. To describe the deviations, it is convenient to make a time change $X^\ve(t) = \tilde{X}^\ve(t/\ve)$. This yields
\begin{equation}
	\dot{X}^\ve(t) = \frac{1}{\ve}\overline{\nabla}H(X^\ve(t)) + b(X^\ve(t)), \quad X^\ve(0) = x. \label{eq.e5}
\end{equation}

If $n=1$, $\lim H(x) = \infty$, and $H(x)$ has just one minimum, $H(X^{\ve}_t)$ describes the evolution of the slow component of $X_t^{\ve}$. The classical averaging principle allows to calculate $\lim_{\ve \downarrow 0} H(X_t^{\ve})$ for any $t$. But, even in the $n=1$ case, the standard averaging does not allow to describe the behavior of $H(X^{\ve}_t)$ as $\ve \downarrow 0$ for large enough $t$ if $H(x)$ has saddle points. The problem should be, in a sense, regularized. For the case of one degree of freedom, this was done in \cite{bib2,bib9}. A noise of small intensity $\delta$ was introduced in system (\ref{eq.e5}). Then, the solution and its slow component became stochastic processes depending on $\ve$ and $\delta$. It was shown that there exists the double limit of the slow component as, first, $\ve \downarrow 0$ and then $\delta \downarrow 0$. A similar approach is used here in the multidimensional case.

In reality, we always deal with systems perturbed by a small noise: the initial point can be perturbed by a small random variable and/or the noise can be added to the equation. Let a white-noise-type perturbation be added to the equation. Moreover, let the noise consist of two components: Component 1 is reversable and preserves the energy (the Hamiltonian $H(x))$, Component 2 destroys the energy integral $H(x)$.

More precisely, instead of dynamical system (\ref{eq.e5}), we observe the diffusion process $X^{\ve,\vk,\delta}(t)$ in $\R^{2n}$ governed by the operator $L^{\ve,\vk,\delta}$,
\begin{align}
	L^{\ve,\vk,\delta}u(x) &= \frac{1}{\ve} \left[ \overline{\nabla}H(x) \cdot \nabla u(x) + \frac{\vk}{2} \text{div} \left( a^{(1)}(x) \nabla u(x) \right)\right] \notag \\
	&\quad + \delta \left[ \frac{1}{2} \text{div} \left( a^{(2)}(x) \nabla u(x) \right) + \beta(x) \cdot \nabla u(x) \right] + b(x) \cdot \nabla u(x) \notag \\
	&= \frac{1}{\ve}\left[ \overline{\nabla}H \cdot \nabla u + \vk L_1 u \right](x) + \delta L_2 u(x) + b(x) \cdot \nabla u(x) \label{eq.e6}
\end{align}
Here $a^{(1)}(x)$ and $a^{(2)}(x)$ are $2n \times 2n$-matrices with smooth and bounded entries, $\beta(x)$ is a smooth vector field in $\R^{2n}$. Parameters $\vk$ and $\delta$ characterize the intensities of random perturbations. We assume that 
\begin{enumerate}
	\item $H(x)$ is smooth enough, has a finite number of critical points that are non-degenerate, and $\lim_{\lvert x \rvert \rightarrow \infty} H(x) = \infty$
	\item $a^{(1)}(x)\nabla H(x) \equiv 0$ for $x \in \R^{2n}, a^{(1)}(x)e \cdot e \geq \lambda^{(1)} \lvert e \rvert^2$ for each $e$ such that $e \cdot \nabla H(X)=0$ with some $\lambda^{(1)}(x) > 0$ if $x$ is not a critical point.
	\item $a^{(2)}(x)$ is positive definite, smooth, and bounded for $x \in \R^{2n}$.
\end{enumerate}
Assumptions on the coefficients of the operator $L^{\ve, \vk, \delta}$ are formulated in Section 3 in more detail. 

One can check that, under the above assumptions with respect to the matrix $a^{(1)}(x)$, the process $X^{\ve,\vk}(t)$ in $\R^{2n}$ governed by the operator $\frac{1}{\ve}\left[ \vk L_1 + \overline{\nabla} H(x) \cdot \nabla \right]$ preserves $H(x)$: $H(X^{\ve,\vk}(t)) \equiv H(x)$ for all $t \geq 0$ with probability 1 (see \cite{bib7}).

We observe the process $X^{\ve, \vk, \delta}(t)$ on a time interval $[0,T]$. It has a fast component that is roughly a motion on connected components $C_k(z)$ of the level set $C(z) = \{ x \in \R^{2n} : H(x) = z \}$, and a slow component that can be described by a motion on the graph $\Gamma$ obtained after identification of points of each connected component of the Hamiltonian (see more details in \cite{bib7} and in the Section 2). 

The fast component on $C_k(z)$ can be approximated for small $\delta > 0$ by the process $X^{\ve, \vk}_*(t)$, $X^{\ve, \vk}_*(0) = x \in C_k(z)$. The latter process has a unique invariant probability measure on $C_k(z)$. If $C_k(z)$ does not contain critical points, this measure has the density $$m_{z,k}(x) := \left( \lvert \nabla H(x) \rvert \oint_{C_k(z)} \frac{dl}{\lvert \nabla H(x) \rvert} \right)^{-1}, x \in C_k(z),$$ and the process $X^{\ve,\vk}(t)$ on $C_k(z)$ is ergodic. This allows us to use the standard averaging inside the edges of $\Gamma$ and describe the limiting behavior as $\ve \rightarrow 0$ of the slow component inside each edge \cite{bib11}. Since $m_{z,k}(x)$ is independent of $\vk > 0$, the limiting slow component is also independent of $\vk$.

Let $Y:\R^{2n} \rightarrow \Gamma$ be the projection of $\R^{2n}$ onto $\Gamma$: $Y(x)$ with $x \in \R^{2n}$ is the point on $\Gamma$ corresponding to the level set component containing $x$. Let us number the edges of $\Gamma$. Then the slow component $Y(X^{\ve,\vk,\delta}(t)) = (z^{\ve,\vk,\delta}(t),k^{\ve,\vk,\delta}(t))$ where $z^{\ve,\vk,\delta}(t) = H(X^{\ve,\vk,\delta}(t))$, $k^{\ve,\vk,\delta}(t)$ is the number of the edge containing $Y(X^{\ve,\vk,\delta}(t))$.

It turns out that the exterior vertices (vertices corresponding to the extremums of $H(x)$) are not accessible in finite time for process $X^{\ve,\vk,\delta}(t)$, but the interior vertices (vertices corresponding to the saddle points) are accessible. This means that the standard averaging principle is not sufficient for description of the limiting slow motion. One should describe what the trajectory of the limiting slow motion is doing after hitting an interior vertex. The case of Hamiltonian systems with one degree of freedom $(n=1)$, $\vk=0$, and $b(x) = 0$ was considered in detail (see \cite{bib9} and the references there). We calculated there the gluing conditions, at interior vertices that, together with standard averaging inside the edges, uniquely describe the limiting slow motion. It turns out that these gluing conditions are defined by only the coefficients $a^{(2)}(x)$; they are independent of the drift.

These results allowed to calculate, in the case of one degree of freedom, the double limit $\overline{y}(t)$ of the slow component $Y(X^{\ve,0,\delta}(t))$ as, first $\ve \rightarrow 0$ and then $\delta \rightarrow 0$ \cite{bib2}. It turns out that the limiting process $\overline{y}(t)$ on $\Gamma$ is deterministic inside the edges, but at certain interior vertices (three edges are attached to each interior vertex), we will have a residual stochasticity: the process $\overline{y}(t)$ enters such a vertex along one of attached edges and leaves the vertex along the other two edges with a non-trivial probability distribution. So that the process that we observe for small $\ve$ and $\delta$ will have certain stochasticity. Moreover, the distribution between these two exit edges is independent of the noise: it will be the same for various matrices $a^{(2)}(x)$ (non-degenerated) and fields $\beta(x)$. If just the initial point is perturbed, say $X^{\ve, \vk, \delta}(0)$ is distributed uniformly in the $\delta$-neighborhood of $x$, the double limit of the slow component as first $\ve \rightarrow 0$ and then $\delta \rightarrow 0$ may not exist. But it exists under reasonable additional assumptions and, if the double limit exists, it coincides with the double limit described above. This means that stochasticity of the long-time behavior of $\tilde{X}^{\ve}(t)$ as $\ve \rightarrow 0$ in the case of one degree of freedom is an intrinsic property of the deterministic system $\tilde{X}^\ve(t)$ defined by (\ref{eq.e4}). This stochasticity is actually caused by the instabilities in the Hamiltonian system.

The goal of this paper is to show that a similar effect can be observed in the case of multidimensional systems with a conservation law. We consider in detail the case of Hamiltonian systems. The general systems with conservation law are shortly considered in the last section.

In the case of one degree of freedom, under our assumptions, one ergodic invariant probability measure of the non-perturbed system is concentrated on each connected level set component. Therefore, even if $\vk=0$, the standard averaging principle allows us to calculate the limiting slow motion for $\ve \rightarrow 0$ inside the edges of $\Gamma$. In the case of many degrees of freedom, as a rule, there are many invariant probability measures on the connected level sets components. Therefore, we have to consider \emph{qualified} random perturbations -- perturbations consisting of two parts: one of intensity $\vk$ which reserves the first integral and another one of intensity $\delta$ which destroy the first integral. We assume that the first part is non-degenerate on the level set components not containing critical points, and that $\delta \ll \vk$ (see the description of qualified perturbations in Section 3).

One should note that the perturbations can have different origin, and consideration of qualified perturbations is natural. For instance, fluctuations of the exterior magnetic field are not changing the energy $H(x)$. Another example: perturbations of Landau-Lifshitz equation (see, for instance, \cite{bib5} and references there).

The next question concerns the structure of the graph in multidimensional case. If $x \in \R^2$, each non-degenerate critical point of $H(x)$ is either an extremum of $H(x)$, in this case just one edge attached to corresponding vertex of $\Gamma$, or a saddle point which corresponds to a vertex with 3 attached edges. If $x \in \R^d$ with $d > 2$, the situation is more sophisticated. It is described in Section 2. 

Limiting behavior of the process $\mathcal{Y}^{\ve, \vk, \delta}(t) = Y(X^{\ve, \vk, \gamma}(t))$ on the graph $\Gamma$ as $\ve \rightarrow 0$ and $\vk, \delta > 0$ are fixed is considered in Section 3. The process $\mathcal{Y}^{\ve, \vk, \delta}(t)$ converges weakly in the space of continuous function $\varphi: [0,T] \rightarrow \Gamma, 0 < T < \infty$ to a diffusion process $Y^{\delta}(t)$ on $\Gamma$. The limit $\mathcal{Y}(t)$ of $\mathcal{Y}^{\delta}(t)$ as $\delta \rightarrow 0$ is considered in Section 4. We will see that the stochastic process $\mathcal{Y}(t)$ on $\Gamma$ will be deterministic inside each edge of $\Gamma$. But when trajectory of $\mathcal{Y}(t)$ hits some of interior vertices it can go to other edges attached to this vertex with certain non-trivial distribution. This distribution will be the same for various qualified random perturbations. In this restricted sense the residual stochasticity is an intrinsic property of deterministic system (\ref{eq.e4}) in the multidimensional case.

In the last section, some generalizations of problems studied in the previous sections are considered briefly. In particular, the metastable behavior of the perturbed system on time intervals growing together with $\delta^{-1}$ and some non-Hamiltonian systems with a first integral, briefly considered there.

Small deterministic perturbations of stochastic systems are also considered briefly in Section 5. If the stochastic systems has a first integral $H(x)$, the long-time evolution of the perturbed system can be described by a stochastic process on the Reeb graph related to $H(x)$. This stochastic process is deterministic inside the edges and has a stochastic behavior at some interior vertices.

Finally, a general remark: as was explained in \cite{bib4}, one of the main characteristics of the long-time evolution of a perturbed system is given by the limiting motion (in an appropriate time scale) of the projection of perturbed process on the simplex $\mathcal{M}$ of invariant probability measures of the non-perturbed system. First integrals are, sometimes but not always, a convenient tool to describe this simplex. For instance, no smooth first integral can be introduced for systems with several asymptotically stable regimes \cite{bib9} or for some area-preserving systems on torus \cite{bib3}. In our case, the points of the graph parametrize the ergodic measures of the system -- extreme points of $\mathcal{M}$. The Markov process $\mathcal{Y}(t)$ induces a dynamical system on $\mathcal{M}$. 

\section{Structure of the Reeb Graph}

Let $H(x), x \in \R^d,$ be a twice continuously differentiable function and $\lim_{\lvert x \rvert \rightarrow \infty}H(x) = \infty$. Assume that $H(x)$ has a finite number of critical points $O_1, ..., O_m$ and that they are non-degenerate. The latter means that the Hessian matrix $\left( \frac{\partial^2 H}{\partial x_i \partial x_j} (O_k) \right)$ is a non-degenerate $d \times d$ matrix. Moreover, let each level set of $H(x)$ has a finite number of connected components. Such functions are called \emph{Morse functions}. 

Consider the level set $C(z) = \{x \in \R^d : H(x) = z\}$. Let $n(z)$ be the number of connected components $C_k(z)$ of $C(z): C(z) = \cup_{k=1}^{n(z)} C_k(z), C_i(z) \cap C_j(z) = \emptyset$ if $i \neq j$. Identify points of each connected component $C_k(z)$ for various $z$. The set obtained after such an identification is homeomorphic to a graph $\Gamma = \Gamma_H$. This graph is called the \emph{Reeb graph} of $H$. Vertices of the graph correspond to the critical points. We write $I \sim O_k$ if an edge $I \subset \Gamma$ is attached to the vertex $O_k$. The number of edges attached to a vertex $O_k$ is called the order of vertex $O_k$. The Reeb graphs arise in the various branches of mathematics (see \cite{bib10} and references there). In particular, the Reeb graphs are useful in asymptotic problems related to the averaging principle \cite{bib9}. To the best of our knowledge, graph $\Gamma_H$ was introducted first in the paper of Adelson-Velskii, and Krourod \cite{bib1} in 1945. The paper of Reeb was published a bit later \cite{bib12}.

Let $Y:\R^2 \rightarrow \Gamma$ be the identification map: $Y(x)$ is the point of $\Gamma$ corresponding to the level set component containing the point $x \in \R^2$. We can number the edges of $\Gamma$ and introduce a coordinate system on $\Gamma$: coordinates of a point $y \in \Gamma$ is a pair $(H,i)$, where $H$ is the value of the function $H(x)$ on $Y^{-1}(y)$, and $i$ is the number of the edge containing $y$.

If $H(x)$ is a function on a Euclidean space $\R^d$, the Reeb graph $\Gamma$ is a tree. Since we assume that $\lim_{\lvert x \rvert \rightarrow \infty} = \infty$, we draw the tree with the root up (see Figure \ref{fig2}). Let $k = k(O)$ where $O$ is one of the critical points of $H(x)$ be the number of negative eigenvalues of the Hessian, $d - k$ is the number of positive eigenvalues. If $k(O) = 0$ or $k(O) = d$ the point $O$ is a local minimum or maximum of $H(x)$ respectively, and the order of the corresponding vertex of $\Gamma_H$ is 1. If $1 \leq k(O) \leq d-1, d > 1$, then $O$ is a saddle point. It is shown in \cite{bib7} that, in the case $1 < k(O) < d-1$, the corresponding vertex has order 2: coordinate $H$ is increasing along one of the edges attached to $O$ and decreasing along another one \cite{bib7}. If $k(O) = 1$ or $k(O) = d-1$, the order of $O \in \Gamma$ can be 2 or 3: any small enough neighborhood of $O \in \R^d, d > 1$ is divided by the surface $C(H(O))$ in three connected parts. Say, if there is just one negative eigenvalue of the Hessian at $O$, then the neighborhood of $O$ has two parts where $H(x) < H(0)$ and one part where $H(x) > H(0)$. But in the case $d > 2$, the parts of the neighborhood where $H < H(0)$ can unite outside of the neighborhood of $O$ (Figure 1a), then the corresponding vertex has order 2. If the parts are not united (Figure 1b), the corresponding vertex has order 3: at one edge $H > H(0)$ and two edges with $H < H(0)$ (see \cite{bib7}).

\begin{figure}[t]
	\center{\includegraphics[scale=0.25]{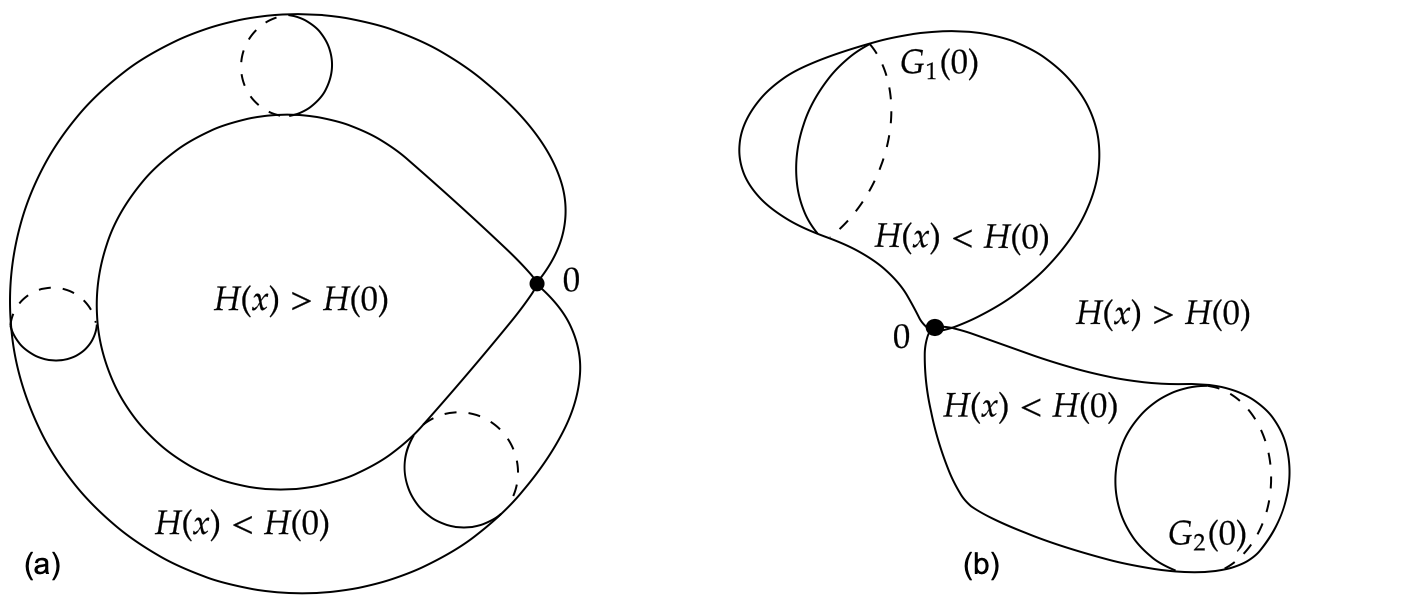}}
    \caption{\label{fig.d1}}
\end{figure}

We will say that a vertex $O \subset \Gamma$ has the type $\frac{i}{j}$ if $i$ edges with $H > H(0)$ and $j$ edges with $H < H(0)$ are attached to 0. For $d > 1$ every vertex 0 of $\Gamma$ belongs to one of five types: $\frac{1}{0},\frac{0}{1},\frac{1}{1},\frac{1}{2},$ and $\frac{2}{1}$. Vertices of type $\frac{1}{0}$ correspond to local minimum points of the function $H$, and $\frac{0}{1}$ to local maximum points; and vertices of types $\frac{1}{1}, \frac{1}{2}, $ and $\frac{2}{1}$ correspond to saddle points of $H$.

We will see that just vertices of type $\frac{1}{2}$ and $\frac{2}{1}$ can lead to the stochasticity of system $(5)$, and this depends on the behavior of the perturbation on the level sets containing the critical points of type $\frac{1}{2}$ and $\frac{2}{1}$. The level set component containing such a point $O$ consist of two parts attached to the vertex (see Figure 1b). If $O$ is an equilibrium of $\frac{1}{2}$ type, $H(x) < H(0)$ for $x \in G_1 \cup G_2$, and $H(x) > 0$ outside of the closure of $G_1 \cup G_2$. If $O$ is a point of $\frac{2}{1}$ type, $H(x) > H(0)$ inside $G_1 \cup G_2$ and $H(x) < H(0)$ outside the closure of $G_1 \cup G_2$.

\begin{figure}[h]
	\center{\includegraphics[scale=0.4]{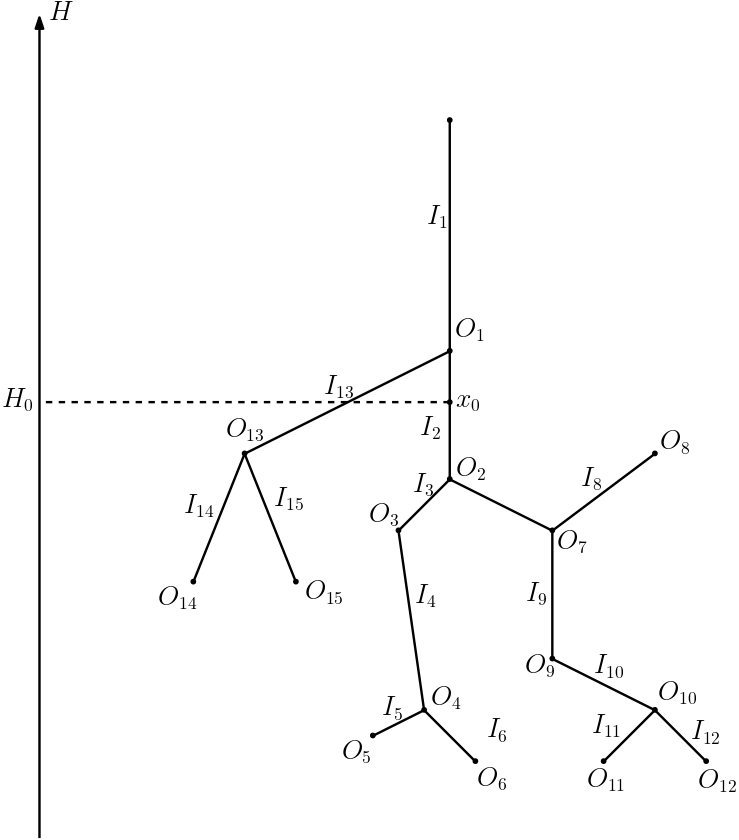}}
    \caption{\label{fig2} The graph $\Gamma$ in Figure 2 consists of 15 edges and 15 vertices. The vertex $O_8$ is of $\frac01$ type and $O_5, O_6, O_{11}, O_{12}, O_{14}, O_{15}$ are of $\frac10$ type. They correspond, respectively, to maximum and minima of $H(x)$. The rest of the vertices correspond to the saddle points of $H(x)$: $O_3$ and $O_9$ are of $\frac11$ type, $O_1, O_2, O_4, O_{10}, O_{13}$ are of $\frac12$ type and $O_7$ is of $\frac21$ type.}
\end{figure}

We say that a vertex $O$ is \emph{essential} for system (5) if it is of $\frac{1}{2}$ type and $\int_{G_i(O)} \text{div} b(x) dx < 0$, $i = 1,2$, or if $O$ is of $\frac{2}{1}$ type and $\int_{G_i} \text{div} b(x) dx > 0$, $i=1,2$.

The following three propositions can be useful to understand the structure of the Reeb graph for a Morse function $H(x), x \in \R^d, \lim_{\lvert x \rvert \rightarrow \infty} H(x) = \infty$.

\begin{proposition}
	Let $\Gamma = \Gamma_H$ be the Reeb graph of a Morse function $H(x), x \in \R^d, \lim_{\lvert x \rvert \rightarrow \infty} H(x) = \infty, d > 1$. Then $\Gamma$ is a tree. Each vertex of $\Gamma$ has order 1,2 or 3 and has one of the five types $\frac{0}{1}, \frac{1}{0}, \frac{1}{1}, \frac{1}{2}$, and $\frac{2}{1}$.
\end{proposition}
\begin{proof}
	The first statement of the proposition is well known: it follows from the fact that the first Batty number $\beta_1(\Gamma)$ of the set $\Gamma$ is less than or equal to the first Batty number of $\R^d$. The second statement is proved in \cite{bib7}. Note that the Reeb graph of a Morse function on a space with more sophisticated topology than $\R^d$, for instance, on a torus, is not necessarily a tree.
\end{proof}

\begin{proposition}
	Let $\Gamma = \Gamma_H$ be the Reeb graph of a Morse function $H(x), x \in \R^d, \lim_{\lvert x \rvert \rightarrow \infty} H(x) = \infty$. Then,
	\begin{enumerate}
		\item[(a)] the number of vertices $\# (\frac{1}{2}, \Gamma)$ of the type $\frac{1}{2}$ in $\Gamma$ is equal to the number $\# (\frac{1}{0}, \Gamma)$ of the vertices of type $\frac{1}{0}$ in it, minus 1 ($\# (\frac{1}{2}, \Gamma)$ is equal to the number of local minima of $H(x)$ minus 1);
		\item[(b)] the number $\# (\frac{2}{1}, \Gamma)$ of vertices of the types $\frac{2}{1}$ in $\Gamma$ is equal to the number of local maxima $\#(\frac{0}{1}, \Gamma)$ of $H(x)$.
	\end{enumerate}
\end{proposition}
\begin{proof}
	(a) Let us prove our statement by induction in the number of edges. If the number of edges is equal to 1, the lower vertex of this edge is of type $\frac{1}{0}$, and there are no vertices of type $\frac{1}{2}$: $\#(\frac{1}{2}, \Gamma) = 0 = \# (\frac{1}{0}, \Gamma) - 1$. 
	
	Now suppose that our statement is true for all graphs with less than $n$ edges. Let the graph $\Gamma$ have $n$ edges. Let $O$ be the vertex corresponding to the minimum of $H(x)$. Then $O$ is a vertex of the type $\frac{1}{0}$. Let $I$ be the edge whose lower end is $O$.
	
	There are five possibilities: the edge $I$ has no vertex being its upper end (the edge goes to infinitely large values of $H$; or it has as its upper end a vertex $o_1$ of the type $\frac{1}{1}, \frac{1}{2}$, or $\frac{2}{1}$.
	
	In the first case, there is only one edge in the graph $\Gamma$. This case is already considered. In the case of $o_1$ being of the type $\frac{1}{1}$, we delete from the graph the edge $I$ together with its lower end $o$, obtaining a graph $\Gamma^1$ with a smaller number of edges. Since $\#(\frac{1}{0}, \Gamma) = \#(\frac{1}{0}, \Gamma^1$ and $\# (\frac{1}{2}, \Gamma) = \# (\frac{1}{2}, \Gamma^1)$ our statement is true.
	
	If $o_1$ is of the type $\frac{1}{2}$, we also delete from $\Gamma$ the edge $I$ with its lower end $O$, obtaining the graph $\Gamma^1$. The vertex $O_1$ goes from type $\frac{1}{2}$ to the type $\frac{1}{1}$ in the graph $\Gamma^1$; $\# (\frac{1}{2}, \Gamma) = \# (\frac{1}{2}, \Gamma^1) + 1$, $\# (\frac{1}{0}, \Gamma) = \# (\frac{1}{0},\Gamma^1) + 1$, and our statement is true for the graph $\Gamma$.
	
	Finally, suppose $O_1$ is of the type $\frac{2}{1}$. Let $I_1$ and $I_2$ be the edges coming to the vertex $O_1$ from above. Delete the edge $I$ together with its lower end $O$ and split the vertex $O_1$ in two: the vertex that we will still denote $O_1$ which we attach as its lower end to the edge $I_1$ and $O_2$ attached as the lower end to the edge $I_2$. We get two disconnected tree graphs: $\Gamma^1$ containing $I_1$ and $O_1$, and $\Gamma^2$ containing $I_2$ and $O_2$ (these graphs are disconnected because the graph $\Gamma$ was a tree). We have: $\#(\frac{1}{2}, \Gamma) = \# (\frac{1}{2}, \Gamma^1) + \# (\frac{1}{2}. \Gamma^2)$, $\# (\frac{1}{0}, \Gamma) = \# (\frac{1}{0}, \Gamma^1) + \# (\frac{1}{0}, \Gamma^2) + 1$, $\#(\frac{1}{2}, \Gamma^i) = \# (\frac{1}{0}, \Gamma^i) - 1$ for $i=1,2$ and thus $\# (\frac{1}{2}, \Gamma) = \# (\frac{1}{0}, \Gamma) - 1$.
	
	The proof of (b) is the same except we have to take into account that there is exactly one edge with no upper end (corresponding to values of $H$ going to infinity). 
\end{proof}

\begin{proposition}
	Let $h(p), p \in \R^{\ell}$, and $F(q), q \in \R^{d-\ell}$ be Morse functions. Let $H(p,q) = F(q) + h(p)$. Assume that $\min_{p \in \R^{\ell}} h(p) = h(p^*) = 0$. Then the critical point $(p^*, q^*)$ corresponding to a vertex $O \in \Gamma_H$ of $\frac{1}{2}$ type if and only if $q^* \in \R^{d - \ell}$ corresponds to $\frac{1}{2}$ type vertex of $\Gamma_F$.
\end{proposition}
\begin{proof}
	It is clear that if $p^*$ is a minimum point of $h(p)$ and $q^*$ corresponds to a vertex of $\Gamma_F$ of $\frac{1}{2}$ type, then just one eigenvalue of the matrix of second derivatives of $H(p,q)$ at the point $(p^*, q^*)$ is negative. Therefore, the component $\mathcal{E}$ of the set $\{(p,q) \in \R^d : H(p,q) < H(p^*,q^*)\}$ attached to $(p^*, q^*)$ consists of one or two connected components attached to the point $(p^*, q^*)$ and the vertex of $\Gamma_H$ corresponding to $(p^*, q^*)$ is of $\frac{1}{1}$ or $\frac{1}{2}$ type (see Figures 1a and 1b). If $\mathcal{E}$ consists of just one component, any two point $A,B \in \mathcal{E}$ can be connected by a continuous curve $\gamma = \gamma_{AB}$ situated inside $\mathcal{E}$. The projection $\tilde{ \gamma }$ of $\gamma$ on the space $\R^{d - \ell}$ connects the projections $A$ and $B$ on $\R^{d - \ell}$. Due to our assumptions on $h(p)$, $\tilde{ \gamma }$ belongs to the component of $\{ q \in \R^{d - \ell} : F(q) < F(q^*) \}$ attached to $q^* \in \R^{d - \ell}$. Thus, if $q^* \in \R^{d-\ell}$ corresponds to $\frac{1}{2}$ type vertex of $\Gamma_F$, the point $(p^*,q^*)$ corresponds to $\frac{1}{2}$ type vertex of $\Gamma_H$.
	
	On the other hand, if $(p^*, q^*)$ corresponds to a $\frac{1}{2}$ type vertex of $\Gamma_H$, then $q^*$ corresponds to a $\frac{1}{2}$ type vertex of $\Gamma_F$: if this is not true, then $q^*$ corresponds to a $\frac{1}{1}$ type vertex of $\Gamma_F$ and the component of $\{ q \in \R^{d - \ell} : F(q) < F(q^*) \}$ attached to a point $q^* \in \R^{d-\ell}$ is a connected set. This implies that the $\mathcal{E}$ is also connected what contradicts to our assumption that $(p^*, q^*)$ corresponds to a $\frac{1}{2}$ type vertex.
\end{proof}

\begin{corollary}
		If $H(p,q) = \frac{\vert p \vert^2}{2} + F(q)$ and $F(q)$ is a Morse function, $\lim_{\vert q \vert \rightarrow \infty} F(q) = \infty$, then $(p^*, q^*)$ corresponds to a $\frac{1}{2}$ type vertex if and only if $p^* = 0$ and $q^*$ corresponds to $\frac{1}{2}$ type vertex of $\Gamma_F$.
\end{corollary}

\section{Convergence to a Diffusion Process on $\Gamma$}
Let $H(x), x = (p,q) \in \mathbb{R}^{2n}$, be a three times continuously differentiable Morse function, with $\lim_{\lvert x \rvert \rightarrow \infty} H(x) = \infty$. Let $a^{(1)}{(x)}$ be a non-negative definite $2n \times 2n$ matrix with twice differentiable entries and bounded derivatives. Let $a^{(2)}(x), x \in \R^{2n}$ be a positive definite $2n \times 2n$-matrix with entries having continuous and bounded second derivatives, $b(x)$ and $\beta(x)$ are bounded together with their derivatives

Assume the Hamiltonian $H(x)$ and the matrix $a^{(1)}(x)$ satisfy the conditions:
\begin{enumerate}
    \item[(A1)] $H(x) \geq c_1 \mid x \mid^2, \mid \nabla H(x) \mid \geq c_2 \mid x \mid, \mid \Delta H(x) \mid \geq c_3$ for sufficiently large $\mid X \mid$, where $c_1, c_2, c_3$ are positive constants
	\item[(A2)] Each connected level set component of $H(x)$ contains at most one critical point.
	\item[(A3)] $a^{(1)}(x) \nabla H(x) = 0$ for all $x \in \mathbb{R}^{2n}$
	\item[(A4)] $\lambda_1(x) \mid e \mid^2 \leq (a^{(1)} e, e) \leq \lambda_2(x) \mid e \mid^2$ for each $e \perp \nabla H(x)$, where $\lambda_1(x) > 0$ if $\nabla H(x) \neq 0$, $\lambda_2(x) < K < \infty$ for all $x$ from a bounded domain $D \in \mathbb{R}^{2n}$ containing all critical points of $H(x)$: if $x_0$ is a critical point of $H(x)$, then there exists constants $k_1, k_2 > 0$ such that $\lambda_1(x) \geq k_1 \mid x-x_0 \mid^2, \lambda_2(x) \leq k_2 \mid x-x_0 \mid^2$ for all $x$ from a neighborhood of $x_0$
	\item[(A5)] Let $\lambda_{k i}$ for $i = 1,2, \dots, 2n$ be the eigenvalues of the Hessian 
	$\left( \frac{\partial^2 H(x)}{\partial x_i \partial x_j} \right)$ 
	at the critical points $x_k$ for $k=1, \dots , m$ and $\lambda^* = \max_{k,i} \lambda_{k i}$. Assume that $\varkappa < (K \lambda^*)^{-1}$.
\end{enumerate}

Denote by $X^{\varepsilon}(t) = X^{\varepsilon, \varkappa, \delta}(t)$ the diffusion process in $\mathbb{R}^{2n}$ governed by the operator $L^{\varepsilon, \varkappa, \delta}$:
\begin{align*}
	L^{\varepsilon, \varkappa, \delta} u(x) = \frac{1}{\varepsilon}\left[ \overline{\nabla} H(x) \cdot \nabla u(x) + \frac{\varkappa}{2} \text{div}(a^{(1)}(x) \nabla u(x)) \right] + b(x) \cdot \nabla u(x) \\
	+ \delta \left[ \frac{1}{2} \text{div} (a^{(2)}(x) \nabla u(x)) + \beta(x) \cdot \nabla u(x) \right] \\
	= \frac{1}{\varepsilon}L_1 u(x) + b(x) \cdot \nabla u(x) + \delta L_2 u(x)
\end{align*}
where $\varepsilon, \varkappa, \delta$ are positive constants.

The process $X^{\varepsilon}(t)$ can be described by the stochastic differential equation
\begin{align}
	\dot{X}^{\varepsilon}(t) = \frac{1}{\varepsilon} \overline{\nabla} H(X^{\varepsilon}(t)) + \sqrt{\frac{\varkappa}{\varepsilon}} \sigma^{(1)}(X^{\varepsilon}(t)) \dot{W}^{(1)}_t + b(X^{\varepsilon}(t))  \notag \\
	+ \sqrt{\delta} \sigma^{(2)}(X^{\varepsilon}(t)) \dot{W}^{(2)}_t + \delta(\tilde{\beta}(X^{\varepsilon}(t)) + \beta(X^{\varepsilon}(t))) \label{eq7}
\end{align}
Here, $W_t^{(1)}$ and $W_t^{(2)}$ are independent Wiener processes in $\mathbb{R}^{2n}$, with $(\sigma^{(1)}(x))^2 = a^{(1)}(x), (\sigma^{(2)}(x))^2 = a^{(2)}(x)$,
\begin{align*}
	\tilde{b}(x) = (\tilde{b}_1(x), \dots, \tilde{b}_{2n}(x)), \quad \tilde{b}_k(x) = \sum_{i=1}^{2n} \frac{\partial a_{ik}^{(1)}(x)}{\partial x_i}, \\
	\tilde{\beta}(x) = (\tilde{\beta}_1 (x), \dots, \tilde{\beta}_{2n}(x)), \quad \tilde{\beta}_k(x) = \sum_{i=1}^{2n} \frac{\partial a_{ik}^{(2)}(x)}{\partial x_i}.
\end{align*}
Note that, if a non-negative definite matrix $a(x)$ has bounded second derivatives, there exists a Lipschitz 
continuous square root of $a(x)$. We can thus assume that $\sigma^{(1)}(x)$ and $\sigma^{(2)}(x)$ are Lipschitz continuous.

The process $X^{\varepsilon}(t) = X^{\varepsilon, \varkappa, \delta}(t)$ for $\varepsilon \ll 1$, $\varkappa$ and $\delta$ fixed, has a fast and slow components. Let $\Gamma = \Gamma_H$ be the Reeb graph for $H(x), Y(X)$ be the corresponding projection of $\R^{2n}$ on $\Gamma$, and $X^{\ve}(0) = x \in Y^{-1}(z,k)$ where $(z,k)$ is an interior point of an edge $I_k \subset \Gamma$. Then, the fast component of $X^{\ve}(t)$ on a small time interval can be, in a sense, approximated by the diffusion process on the manifold $C_k(z) = Y^{-1}(z,k)$ governed by the operator $\frac{1}{\ve}L_1$. This diffusion process has the unique invariant probability measure with the density. 
$$m_{z,k}(x) = \left( \oint_{C_k(z)} \frac{dV}{\mid \nabla H(y) \mid} \mid \nabla H(x) \mid \right)^{-1}$$
with respect to the volume element $dV$ on $C_k(z)$. If $C_k(z)$ contains a critical point, the invariant measure is also unique and is concentrated at the critical point.

The slow component of $X^{\varepsilon}(t)$ changes on the Reeb graph $\Gamma_H$ corresponding to the Hamiltonian. This slow component is
$$\Y^{\varepsilon, \delta}(t) = \Y^{\varepsilon, \varkappa, \delta}(t) = Y(X^{\varepsilon, \varkappa, \delta}(t)).$$

Our goal in this section is to prove that the measures induced in the space of continuous functions $\varphi : [0,T] \rightarrow \Gamma$ by the process $\Y^{\varepsilon, \delta}(t)$ converge weakly to a measure corresponding to a diffusion process $\Y_t^{\delta}$ on $\Gamma$ and to calculate the characteristics of this limiting diffusion process.

It  was proved in \cite{bib7} that if the assumptions A1 - A5 are satisfied and $$b(x) \equiv \beta(x) \equiv 0, x \in \mathbb{R}^{2n},$$ the distribution in the space $C_{0,T}$ of continuous functions $[0,T] \rightarrow \Gamma$ corresponding to the process $\Y(t)^{\varepsilon, \delta}$ converges as $\varepsilon \rightarrow 0$ to the distribution induced by $a$ diffusion process $\tilde{\Y}^{\delta}(t)$ on $\Gamma$. The process $\tilde{\Y}^{\delta}(t)$ is defined as follows: inside an edge $I_k \subset \Gamma$, the process $\tilde{\Y}^{\delta}(t)$ is governed by the operator
\begin{align}
	\tilde{\ell}_k^{\delta} = \frac{\delta}{2 v_k(z)} \frac{d}{dz} (h_k(z) \frac{d}{dz}), \quad v_k(z) = \oint_{C_k(z)} \frac{ds}{\mid \nabla H(x) \mid} \\
	h_k(z) = \oint_{C_k(z)} \frac{a^{(2)}(x) \nabla H(x) \cdot \nabla H(x)}{\mid \nabla H(x) \mid} ds, \quad k = 1,2, \dots, m,	 \notag
\end{align}
where $ds$ is the volume element on $C_k(z)$. One can check that 
\begin{gather*}
	h_k(z) = \int_{G_k(z)} \text{div}\left( a^{(2)}(x) \nabla H(x) \right) dx, \\
	v_k(z) = \frac{d}{dz} V_k(z)
\end{gather*}
where $G_k(z)$ is the domain in $\mathbb{R}^{2n}$ bounded by $C_k(z)$, $V_k(z)$ is the volume of $G_k(z)$. Put $$\overline{a}^{(2)}_k(z) = \frac{h_k(z)}{v_k(z)};$$ here $\delta \overline{a}^{(2)}_k(z)$ is the diffusion coefficient of the process $\tilde{\mathcal{Y}}^{\delta}(t)$.

To describe the limiting process $\tilde{\Y}^{\delta}(t)$ in a unique way, one should add to the operators $\tilde{\ell}^{\delta}_k$ the domain $D_{\tilde{A}^{\delta}}$ of the generator $\tilde{A}^{\delta}$ of the process $\tilde{\Y}^{\delta}(t)$ (see \cite{bib7,bib9}). A smooth inside the edges function $u(z,k)$ on $\Gamma$ belongs to $D_{\tilde{A}^{\delta}}$ if $u(z,k)$ as well as $\tilde{\ell}^{\delta}_k u(z,k)$ are continuous on $\Gamma$ and $u(z,k)$ satisfies the gluing conditions:
\begin{quotation}
	If $O_i$ is a $\frac11$ type vertex, then $\frac{du(z,k)}{dz}$ should be continuous at $O_k$.
	
	If $O_i$ is of $\frac12$ or $\frac21$ type, the following condition should be satisfied at $O_i$
	\begin{align}
		\sum_{k : I_k \sim O_i} \gamma_{ik} D_k u (O_i) = 0, \gamma_{ik} = \oint_{C_{ik}}\frac{a^{(2)}(x) \nabla H(x) \cdot \nabla H(x) ds}{\lvert \nabla H(x) \rvert} \label{eq9}
	\end{align}
\end{quotation}
Here $C_{ik} = \langle x : Y(x) = O_i \rangle \cap \partial \{ x : Y(x) \in I_k \}$, $D_k u(O_i)$ means differentiation along $I_k$ at vertex $O_i$, $I_k \sim O_i$. The coefficients $\gamma_{ik}$ can be written in the form
\begin{align}
\gamma_{ik} = \int_{G_{ik}} \text{div } \left(a^{(2)}(x) \nabla H(x) \cdot \nabla H(x) \right) dx
\end{align} where $G_{ik}$ is the domain bounded by $C_{ik}$. The operators $\tilde{\ell}_k^{\delta}$ and these gluing conditions determine the limiting slow motion $\tilde{\Y}^{\delta}(t)$ in the unique way.

If the operator $L^{\varepsilon, \varkappa, \delta}$ and corresponding process $X^{\varepsilon, \varkappa, \delta}(t)$ have the drift term $b(x) + \delta \beta(x)$ the limiting process inside the edges will be slightly different, but it can be calculated similarly to the case $b(x) + \delta \beta(x) \equiv 0$: Applying the Ito formula to $H(X^{\varepsilon}(t))$, where $X^{\varepsilon}(t)$ is defined by (\ref{eq7}), and taking into account our assumptions on $a^{(1)}(x)$ and the identity $\nabla H(x) \cdot \overline{\nabla} H(x) \equiv 0$, we get
\begin{align}
	H(X^{\varepsilon}(t)) - H(x) = \sqrt{\delta} \int_0^t \nabla H(X^{\varepsilon}(s)) \cdot \sigma^{(2)}(X^{\varepsilon}(t)) dW_s^{(2)} + \delta \int_0^t L_2 H(X^{\varepsilon}(s)) ds \notag \\
	+ \delta \int_0^t \nabla H(X^{\varepsilon}(s)) \cdot \beta(X^{\varepsilon}(s)) ds + \int_0^t \nabla H(X^{\varepsilon}(s)) \cdot b(X^{\varepsilon}(s)) ds
\end{align}
Using the self-similarity property of the Wiener process, one can find a one dimensional Wiener process $\tilde{W}(t)$ such that the stochastic integral can be replaced by $$\sqrt{\delta} \tilde{W} \left( \int_0^t a^{(2)}(X^{\varepsilon}(s)) \nabla H(X^{\varepsilon}(s)) \cdot H(X^{\varepsilon}(s)) ds \right).$$ This replacement allows us to reduce the averaging of the stochastic integral to averaging of a \emph{non-stochastic} integral.

For each continuous function $f(x), x \in \mathbb{R}^2$, and a level set component $C_k(z)$ not containing critical points, put $$\overline{f}_k(z) = \oint_{C_k(z)} f(x) m_{z,k}(x) dx$$

Define the diffusion process $\Y^{\delta}(t) = (Z_t^{\delta}, k_t^{\delta})$ on $\Gamma$ as the process governed by the operators $\ell_k^{\delta}$, 
\begin{align*}
	\ell_k^{\delta} u(z,k) &= \tilde{\ell}_k u + (\delta \overline{\beta}_k(z) + \overline{b}_k(z) \frac{du}{dz} \\
	&= \frac{1}{2v_k(z)} \frac{d}{dz} \left(h_k(z) \frac{du}{dz}\right) + (\delta \overline{\beta}_k(z) + \overline{b}_k(z)) \frac{du}{dz}
\end{align*}
where $k = 1, \dots, m,$ inside the edges with the same domain as the generator of the process $\tilde{\Y}^{\delta}(t)$.
\vspace{2.5mm}\\
\noindent \textbf{Theorem 1.} \label{thm1} Assume that the assumptions A1 - A5 are satisfied. 
	Then the measures induced by the process $\Y^{\varepsilon, \delta}(t) = Y(X^{\varepsilon, \varkappa, \delta}(t))$ in the space $C_{0,T}$ of continuous functions $\varphi : [0,T] \rightarrow \Gamma$ provided with the uniform topology converge weakly as $\varepsilon \rightarrow 0$ to the measure induced in $C_{0,T}$ by the process $\Y^{\delta}(t), \Y^{\delta}(0) = Y(x)$.
\vspace{2.5mm}\\
The proof of this theorem is based on several lemmas.
\vspace{2.5mm}\\
\noindent \textbf{Lemma 1.} \label{lem1} The measures in $C_{0,T}$ induced by processes $\Y_t^{\varepsilon,\delta}, \varepsilon \rightarrow 0$ are tight in the weak topology.
\vspace{2.5mm}\\
This lemma can be proved similarly to the Lemma 8.3.2 from \cite{bib9}, and we omit the proof.
\vspace{2.5mm}\\
\noindent \textbf{Lemma 2.} \label{lem2} Let $f(x), x \in \mathbb{R}^2, Y(x) \in I_k$, be twice continuously differentiable, $(z,k) \in I_k \subset \Gamma$, $I_k^{\gamma} = \{ (z,k) \in I_k$, distance between the point $(z,k) \in I_k$ and the ends of $I_k$ is greater than $\gamma \}$, where $\gamma$ is small enough. Let $\tau = \tau^{\gamma} = \min \{ t : Y(X^{\varepsilon}(t)) \notin I_k^{\gamma} \}$. Then, for each $T < \infty, Y(x) \in I_k^{\gamma}, k = 1, \dots, m$ and $h > 0$
	\begin{align}
		\lim_{\varepsilon \downarrow 0} P_x \left\{ \max_{0 \leq t \leq T \wedge \tau} \mid H(X^{\varepsilon}(t) - Z_t^{\delta} \mid > h \right\} = 0 \label{eq12}
	\end{align}
\begin{proof}[Proof.]
	It follows from the standard averaging principle \cite{bib9,bib11}, that for any $f \in C^2(\mathbb{R}^{2n})$, $x$ such that $Y(x) \in I_k$, and any $\delta > 0, T < \infty$
	\begin{align*}
	\lim_{\varepsilon \downarrow 0} P_x \left\{ \max_{0 \leq t \leq T} \lvert \int_{0}^{t \wedge \tau} 
	\left( f(X^{\varepsilon}(s)) - \overline{f}_k (Y(X^{\varepsilon}(s))) \right) ds \rvert > \delta \right\} = 0 
	\end{align*}
	This equality together with the properties of the Wiener process implies that each limiting in the uniform topology point $\hat{\mathcal{Y}}_t^{\delta}$ of the family of processes $Y(X^{\varepsilon, \varkappa, \delta}(t))$ on $\Gamma$ as $\varepsilon \rightarrow 0$ satisfies the equation 
	\begin{gather*}
		\dot{\hat{\mathcal{Y}}}_t^{\delta} = \sqrt{\delta} \dot{\tilde{W}} \left( \int_0^t (\overline{a^{(2)} \nabla H \cdot \nabla H}) (\hat{\Y}^{\delta}(s) ds \right) + \delta \overline{L_2 H}(\hat{\Y}^{\delta}(s)
		+ (\overline{\nabla H \cdot b})(\hat{\Y}^{\delta}(t) \\
		\hat{\Y}^{\delta}(0) = Y(x), t < \hat{\tau} = \min \{ t : \hat{Y}_t \notin I_k^{\delta} \}
	\end{gather*}
	where $\tilde{W}(t)$ is an appropriate Wiener process. Here $\overline{f}(z)$ means the averaging of a function $f(x), x \in \R^{2n}$, with respect to the invariant density $m_{z,k}(x)$ on $C_k(z)$. It is easy to check that the distributions of the process $\hat{\Y}^{\delta}(t)$ and $Z_t^{\delta}$, $0 \leq t \leq \tau$ coincide. This implies (\ref{eq12}).
\end{proof}
Denote by $\alpha(x) = \alpha_{\gamma}(x)$ a smooth function such that $0 \leq \alpha(x) \leq 1$, $\alpha(x) = 1$ if $Y(x) \in \cup_{k} I_k^{\gamma}$ and $\alpha(x) = 0$ if $Y(x) \notin \cup_k I_k^{\gamma/2}$. Let $b_{\gamma}(x) = b(x) \alpha(x), \beta_{\gamma}(x) = \beta(x) \alpha(x)$. Denote by $X^{\varepsilon, \delta}_{\gamma}(t)$ the process in $\mathbb{R}^{2n}$ defined by equation (7) with $b$ and $\beta$ replaced by $b_{\gamma}(x)$ and $\beta_{\gamma}(x)$ respectively. 
\vspace{2.5mm}\\
\noindent \textbf{Lemma 3.} \label{lem3} The Theorem \ref{thm1} holds for processes $X_{\gamma}^{\varepsilon,\delta}(t)$ for any small $\gamma > 0$.
\begin{proof}[Proof of Lemma \ref{lem3}.]
	The proof of Lemma \ref{lem3} follows from Lemmas \ref{lem1} and \ref{lem2} and the main result of \cite{bib7}, since the proof of the latter result is based just on the tightness of distributions corresponding $\Y^{\varepsilon, \delta}(t)$ (Lemma \ref{lem1}), uniform convergence in probability inside the edges (Lemma \ref{lem2}), and some estimates inside small neighborhoods of the vertices. But $X_{\gamma}^{\varepsilon}((t)$, inside these neighborhoods, coincides with the process considered in \cite{bib7}, therefore the estimates hold for the process $X_{\gamma}^{\varepsilon}(t)$ as well.
\end{proof}
Consider now a star-shape graph $\Gamma$ consisting of edges $I_1, \dots, I_m$ and one vertex $O$ (Figure \ref{fig3}). Let $z = z(y)$ be the distance from a point $y \in \Gamma$ to the vertex $O$. Let $(z(t), i(t))$ be the diffusion process on $\Gamma$ which is governed by the operator
\begin{align*}
	\ell_i = \frac12 \sigma_i^2(z) \frac{d^2}{dz^2}
\end{align*} inside the edge $I_i \subset \Gamma, i = 1, \dots, m$ and by the gluing condition $$\sum_{i=1}^m \alpha_i D_i u(0) = 0, \alpha_i \geq 0, \sum_{i=1}^m \alpha_i = 1.$$ Here $D_i$ means differentiation in $z$ along the edge $I_i$, functions $u$ and $\ell_i u$ are assumed to be continuous on $\Gamma$.


\begin{figure}
	\center{\includegraphics[scale=0.4]{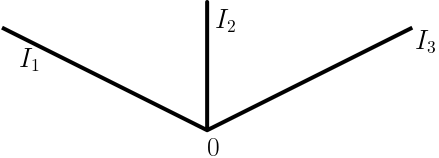}}
    \caption{\label{fig3}}
\end{figure}

Assume that the coefficients of the operators $\ell_i$ are smooth inside the edges 
	\begin{align*}
	\lim_{z \rightarrow 0} \sigma_i^2(z) \mid \ln z \mid = c_i^{(1)} \in (0, \infty), \lim_{z \rightarrow 0} \frac{d \sigma_i^2(z)}{dz} \cdot z \ln^2 z = c_i^{(2)} \in (0, \infty) 
	\end{align*}
Such a diffusion process $(z(t), i(t))$ exists, $O$ is a regular point, and the process spends time zero at $O$ with probability 1.
\vspace{2.5mm}\\
\noindent \textbf{Lemma 4.} \label{lem4} There exists a one dimensional Wiener process $W_t$ and a continuous non-decreasing process $\lambda_t$ which are measurable with respect to the $\sigma$-field  $\mathcal{F}_t$ generated by $(z(s), i(s)), 0 \leq s \leq t$, such that $$\dot{z}(t) = 
\sigma_{i(t)}(z(t),i(t)) \dot{W}_t + \dot{\lambda}_t.$$ The process $\lambda_t$ increases only when $z(t) = 0$.
\begin{proof}
	The proof of this lemma follows from the proof of Lemma 2.2 of \cite{bib6}. It is assumed in that paper that $\sigma_i^2(z)$ in is strictly positive and smooth. Our assumptions allows us to use the same arguments as in \cite{bib6}.
\end{proof}
\noindent \textbf{Lemma 5.} \label{lem5} Let $P$ be the probability measure induced in the space $C_{0,T}$ of continuous functions $\varphi : [0,T] \rightarrow \Gamma$ by the diffusion process $(z(t), i(t))$ defined above, and $W_t$ be the Wiener process introduced in Lemma \ref{lem4}. Assume that a vector field $h(z,i), (z,i) \in \Gamma$, satisfies the condition:
\begin{gather*}
	h(z,i) = \sigma_i(z)e(z,i) \\
	\lvert e(z,i) \rvert \leq M < \infty \text{ for } (z,i) \in \Gamma.
\end{gather*}
	Define the probability measure $\mathcal{Q}$ in $C_{0,T}$ such that $$\frac{d\mathcal{Q}}{dP} = \exp \{ \int_0^T e(z(s), i(s)) dW_s - \frac12 \int_0^T \mid e(z(s), i(s)) \mid^2 ds \}.$$
	Then the measure $\mathcal{Q}$ in $C_{0,T}$ corresponds to the diffusion process on $\Gamma$ governed by the operators $\tilde{\ell}_i = \frac12 \sigma_i^2(z) \frac{d^2}{dz^2} + h(z,i) \frac{d}{dz}$ and the same gluing conditions as for the process $(z(t), i(t))$.
\begin{proof}
	This lemma follows from Lemma 3.3 of \cite{bib6}.
\end{proof}
\begin{proof}[Proof of Theorem \ref{thm1}.]
	Let $\Phi[X(\cdot)]$ be a continuous functional on the space of continuous functions on $[0,T], 0 < T < \infty,$ with values in $\Gamma$, $\mid \Phi[X] \mid \leq M < \infty$. We have 
	\begin{align*}
		\mid E_x \Phi[Y(X^{\varepsilon,\delta}(\cdot)] - E_{Y(x)}\Phi[\Y^{\delta}(\cdot)] \mid \leq \mid E_x \Phi[Y(X^{\varepsilon,\delta)}] - E_x \Phi[Y(X_{\gamma}^{\varepsilon, \delta})] \mid \\
		+ \mid E_{Y(x)} \Phi[Y(X^{\varepsilon,\delta}_{\gamma}] - E_{Y(x)}\Phi[\Y_{\gamma}^{\delta}] \mid + \mid E_{Y(x)}\Phi[Y_{\gamma}^{\delta}] - E_{Y(x)} \Phi[\Y^{\delta}] \mid \\
		= A_1 + A_2 + A_3 \\
	\end{align*}
	To get the upper bound for the first term $A_1$ in the right hand side of this inequality, note that, according to the Girsanov theorem, the processes $X^{\varepsilon,\delta}(t)$ and $X^{\varepsilon, \delta}_{\gamma}(t), X^{\varepsilon, \delta}(0) = X^{\varepsilon, \delta}_{\gamma}(0) = x$ induce in the space of trajectories on the time interval $[0,T]$ measures $\mu$ and $\mu_{\gamma}$ respectively such that $$\frac{d\mu_{\gamma}}{d\mu} = \exp \{ \frac{1}{\sqrt{\delta}} \int_0^T e(X^{\varepsilon,\delta}(s)) dW_s - \frac{1}{2\delta} \int_0^T \mid e(X^{\varepsilon, \delta}(s)) \mid^2 ds \}$$ where $e(x) = e_{\gamma}(x) = [a^{(2)}(x)]^{-1/2}[b(x) - b_{\gamma}(x) + \delta (\beta(x) - \beta_{\gamma}(x)]$. Note that, according to our assumptions, $\mid e(x) \mid$ is a bounded function equal to zero if $Y(x) \in \cup_k I_k^{\gamma}$. Then we get 
	\begin{align*}
		0 \leq A_1 \leq M \cdot E_x \mid \frac{d\mu_{\gamma}}{d\mu} - 1 \mid \leq M \cdot (E_x(\frac{d\mu_{\gamma}}{d\mu} - 1)^2)^{1/2} \\
		\leq E(\exp \{ \frac{1}{\delta} \int_0^T \mid e_\gamma \mid^2 ds \} - 1) \rightarrow 0 \text{ as } \gamma \rightarrow 0
	\end{align*}
	for each $\delta > 0$ uniformly on $\varepsilon$ since $a^{(2)}(x)$ is non-degenerate $b$ and $\beta$ are bounded and the time spent by $X_t^{\varepsilon, \delta}$ during the time interval $[0,T]$ inside the $\gamma$-neighborhood of the set $\{ x \in \R^{2n} : Y(x) \text{ is a vertex} \}$ tends to zero together with $\gamma$.
	
	The second term $A_2$ tends to zero as $\varepsilon \rightarrow 0$ because of Lemma \ref{lem3}.
	
	To show that $A_3$ tends to zero, note that $\overline{b}_i(z) + \delta\overline{\beta}_i(z)$ at each vertex tends to zero as $z$ tends to vertex in the same way or even faster as the diffusion coefficient. Therefore, the function $h(z,i)$ in Lemma \ref{lem5} is bounded. For any $\theta \in (0,1]$, one can find an integer $N$ such that during the time interval $[0,T]$ processes $\Y_{\gamma}^{\delta}(t)$ and the process $\Y^{\delta}(t)$ will make not more than $N$ transitions between different vertices with probability $1-\theta$ for each fixed $\delta > 0$. These properties imply that the measures in the space of trajectories of $\Y^{\delta}(t)$ and $\Y^{\delta}_{\gamma}(t)$ are absolutely continuous with the density given by Lemma \ref{lem5}. Then one can check that $A_3 \rightarrow 0$ as $\gamma \rightarrow 0$ similarly to the estimate of $A_1$.
\end{proof}

\section{Limiting Process on the Graph. Main Results.}
The limiting behavior of the process $\Y_t^{\delta}$ on the graph $\Gamma$ as $\delta \rightarrow 0$ is considered in this section.

Define functions $\hat{b}_i(z)$ and $\overline{b}_i(z)$, $(z,i) \in \Gamma$, as follows
\begin{align*}
	\hat{b}_i(z) = \int_{G_i(z)} \text{div} b(x) dx, \quad \overline{b}_i(z) = \frac{1}{V_i'(z)} \hat{b}_i(z), \quad (z,i) \in \Gamma
\end{align*}
where $G_i(z)$ is the domain in $\R^{2n}$ bounded by the surface $C_i(z) = Y^{-1}(z,i)$, and $V_i(z)$ is the volume of $G_i(z)$.

If $O_0 = (z_0, i)$ is a vertex, $\hat{b}_i(O_0) = \lim_{z \rightarrow z_0}\hat{b}_i(z)$ along $I_i \sim O_0$, $V_i(O_0) = \lim_{z \rightarrow z_0}V_i(z)$ along $I_i$.

We make the following assumption:\\
\textbf{(A6)} There is only a finite number of points $(z,i) \in \Gamma$ where $\hat{b}_i(z) = 0$; $\hat{b}_i(z) \neq 0$ if $(z,i)$ is an interior vertex.

The functions $\hat{b}_i(z)$ and $\overline{b}_i(z)$ are smooth inside the edges and equal to zero at exterior vertices. They can have discontinuities at the vertices of type $\frac12$ and $\frac21$. 
Let $O_0 = (z_0, i) = Y(0)$ be such a vertex and $I_i \sim O_0$, $i=1,2,3$. Then the connected component of $C_i(z_0) = \{ x \in \R^{2n} : H(x) = H(0) \}$ containing the point $O$ bounds a set consisting of two domains $G_1$ and $G_2$ attached to the saddle point $O$ (see Figure 1b). Let $Y^{-1}(I_1) \subseteq G_1, Y^{-1}(I_2) \subseteq G_2$, and $Y^{-1}(I_3) \subset \R^{2n} \setminus (G_1 \cup G_2)$. Then $\hat{b}_3(O_0) = \hat{b}_1(O_0) + \hat{b}_2(O_0)$.

If the trajectory of equation $\dot{z}(t) = \overline{b}_i(z_t)$ started at a point of $I_i$ close to $O_0$ is attracted to $O_0$, we say that $I_i$ is the entrance edge for $O_0$, otherwise it is an exit edge for the vertex $O_0$. It is clear that, if (A6) is satisfied, each vertex of the $\frac12$ or of $\frac21$, type has either two entrance edges and one exit edge, or two exit and one entrance edge.

If $O_0' \in \Gamma$ is an exterior vertex and an edge $I_j \sim O_0'$, then $\lim \overline{b}_j(z) = 0$ as the point $(z,j)$ approaches $O_o'$. Such a vertex $O_0'$ can be asymptotically stable equilibrium of the equation $\dot{z}(t) = b_j(z_t)$ or a repelling equilibrium.

Define a continuous with probability 1 Markov process $\Y(t)$ on $\Gamma$: let, inside each edge $I_i \subset \Gamma, \Y(t)$ be the deterministic motion governed by the equation
\begin{equation}
	\dot{\Y}(t) = \overline{b}_i(\Y(t)) \label{eq13}
\end{equation}

As it follows from our assumptions, if $O_0$ is an exterior vertex and $I_i \sim O_0$, then $\lim_{(z,i) \rightarrow 0} \mid \overline{b^i}(z) z^{-1} \mid < c_i < \infty$ (see Lemma \ref{lem6} below). This implies that an exterior vertex is inaccessible in a finite time for $\Y(t)$. But each interior vertex, according to Lemma \ref{lem6}, is accessible in a finite time. Therefore, to define $\Y(t)$ for all $t > 0$, we should say what $\Y(t)$ is doing after hitting such a vertex. If the interior vertex $O_0$ has just one exit edge attached to $O_0$, then $\Y(t)$ leaves $O_0$ through this exit edge without any delay. If there are two exit edges from $O_0$, say $I_{i_1}$ and $I_{i_2}$ and one entrance edge, then $\Y(t)$ without any delay, goes from $O_0$ to $I_{i_1}$ or $I_{i_2}$ with probabilities $p_{i_1}(O_0)$ and $p_{i_2}(O_0)$ respectively, where 
\begin{equation}
	p_{i_k}(O_0) = \frac{\mid \hat{b}_{i_k}(O_0) \mid}{\mid \hat{b}_{i_1}(O_0) \mid + \mid \hat{b}_{i_2}(O_0) \mid}, \quad k=1,2 \label{eq14}
\end{equation}
Here $\hat{b}_{i_k}(O_0) = \int_{G_k(O_0)} \text{div} b(x) dx$, where $G_1(O_0)$ and $G_2(O_0)$ are two domains in $\R^{2n}$ attached to the saddle point $O = Y^{-1}(O_0)$ introduced above. Such domains correspond to each vertex of type $\frac12$ or $\frac21$ (see Figure 1b). It is easy to see that equation (19) together with described above conditions at the vertices define the stochastic process $\Y(t)$ in the unique way.
\vspace{2.5mm}\\
\noindent \textbf{Theorem 2.} \label{thm2} Let conditions (A1) - (A6) be satisfied. Then the process $\Y^{\delta}(t)$ on $\Gamma$ converges weakly in the space of continuous functions $\varphi:[0,T] \rightarrow \Gamma$ to the process $\Y(t), \Y(0) = \Y^{\delta}(0)$ on each finite time interval $[0,T]$.
\vspace{2.5mm}\\
To prove this theorem, we need several lemmas. We denote by $c_{k,i}$ various nonnegative constants.
\vspace{2.5mm}\\
\noindent \textbf{Lemma 6.} \label{lem6} Assume that the conditions (A1) - (A6) are satisfied, $\Gamma$ is the Reeb graph for the Hamiltonian $H(x), x \in \R^{2n}$, $O_0$ is a vertex of $\Gamma$, and $I_i \sim O_0$. Without loss of generality, we assume that $H(0) = 0$.
	\begin{itemize}
		\item[(a)] If $O_0$ is an exterior vertex, then 
\begin{equation}
 		\lim_{z \rightarrow 0} \frac{\overline{b}_i(z)}{\lvert z \rvert} = c_{2,i} \geq 0,~~\lim_{z \rightarrow 0} \frac{\overline{a}_i^{(2)}(z)}{\lvert z \rvert} = c_{2,i} > 0.
\end{equation}
		\item[(b)] If $O_0$ is an interior vertex, then 
		\begin{equation}
			\lim_{z \rightarrow 0} \overline{b}_i(z) \lvert \ln \lvert z \rvert \rvert = c_{3,i} > 0,~~
			\lim_{z \rightarrow 0} \overline{a}^{(2)}_i(z) \lvert \ln \lvert z \rvert \rvert = c_{4,i} > 0. \label{eq16}
		\end{equation}
	\end{itemize}
\begin{proof}[Proof.]
	Without loss of generality we can assume that in a small neighborhood of the critical point $O = Y^{-1}(O_0)$, $H(x) = \pm x_1^2 \pm x_2^2 \pm \dots \pm x_{2n}^2$. If $O_0$ is an exterior vertex, all signs in this quadratic form are the same. To be specific, let all of them be positive (this means that $\Y^{-1}(O_0)$ is a minimum point of $H(x)$). Then the set $G_z = \{ x \in \R^{2n} : H(x) < z \}$, for small $z$ is the ball $B_{\sqrt{z}}(O)$ of radius $\sqrt{z}$ with the center at $O$. The volume $V(z)$ of this ball is $c_5 \lvert z \rvert^{2n}$ and $V'(z) = 2nc_5z^{2n-1}$. The numerator for $\overline{b}_i(z)$ is of order div $b(0) z^{2n}$. Therefore, $\lvert \overline{b}(z) \rvert \leq c_{1,i} \lvert z \rvert$. If div $\left(b(O)\right) = 0, \lvert \overline{b}(z) \rvert = o(z)$ as $z \rightarrow 0$ and $c_{1,i} = 0$, otherwise $c_{1,i} > 0$. The second statement of (15) can be proved similarly: $c_{2,i} > 0$ since $\overline{a}^{(2)}(z) > 0$. 
	
	Let now $O_0$ be an interior vertex and $I_i \sim O_0$. To be specific, assume that $H(x) \leq 0$ if $Y(x) \in I_i$. According to our assumptions, $\hat{b}_i(O_0) \neq 0$. Denote by $W(z)$ the volume of the domain. $$\mathcal{E}_z = \{ x \in \R^{2n} : -z < H(x) < H(0) = 0 \}.$$
	
	To prove the first statement in (\ref{eq16}), it is sufficient to show that $\lim_{z \rightarrow 0} W'(z) \lvert \ln z \rvert^{-1} = c_6 \in (0, \infty)$. We can assume that $O$ is the origin and in a small enough neighborhood of $O$
	\begin{equation}
		H(x) = x_1^2 + \dots + x_k^2 - x_{k+1}^2 - \dots x^2_{k+l},~~~ k \geq 1, l > 1, k+l = 2n. \label{eq17}
	\end{equation}
	Put 
	\begin{gather*}
		\sum_{i=1}^k x_i^2 = u_1^2,~~
		\sum_{i=k+1}^{k+l}x_i^2 = u_2^2.
	\end{gather*} Let $\mathcal{E}_z^h = \mathcal{E}_z^h(O) = \{ x \in \R^{2n} : u_1^2 < h, u_2^2 < h, u_1^2 - u_2^2 < z \}$, where $h$ is small enough  so that (\ref{eq17}) holds in $\mathcal{E}_z^h$.
	
	Consider the case $k=1,l=2n-1$. Denote by $S(z)$ the area of the domain 
	\begin{gather*}
		\{(u_1,u_2) \in \mathbb{R}^2 : 0 < u_1 < \sqrt{h}, \\
		\lvert u_2 \rvert < \sqrt{h}, u_1^2 - u_2^2 < z\}.
	\end{gather*}
	
	We have $$S(z) = \int_{-\sqrt{h}}^{\sqrt{h}}du_2 \sqrt{z + u_2^2} = h \sqrt{h^2 + z} - \frac{z}{2} \ln z + z h(h +\sqrt{h^2 + z}).$$
	
	One can derive from the last equality, that
	\begin{equation}
		\lim_{z \rightarrow 0} \frac{S'(z)}{\lvert \ln z \rvert} = c_6 > 0. \label{eq18}
	\end{equation}
	
	Denote by $A_m(\rho)$ the area (volume) of the sphere of radius $\rho$ in $\R^m$, $A_m(\rho) = c_7 \rho^{m-1}$. Then, taking into account that the domain $\mathcal{E}_z^h$ is invariant with respect to rotations around the axis $x_1$ for the case $k=1$, we get the following expression for the volume $W_h(z)$ of $\mathcal{E}_z^h$:
	\begin{equation}
		W_h(z) = \int_0^z S'(y) A_{2n-2}(h-y)dy = c_8 \int_0^z S'(y)(h-y)^{2n-2}dy,~~c_8 > 0. \label{eq19}
	\end{equation}
	It follows from (\ref{eq18}) and (\ref{eq19}), that $$\lim_{z \rightarrow 0} \frac{W_h(z)}{\lvert \ln z \rvert} = c_9 > 0.$$
	
	There are no critical points in $\mathcal{E}_z \setminus \mathcal{E}_z^h$, therefore the volume $\tilde{W}(z)$ of this set is a smooth function with a bounded derivative. Thus,
	\begin{equation}
		\lim_{z \rightarrow 0} \frac{W'(z)}{\lvert \ln z \rvert} = \lim_{z \rightarrow 0} \frac{W_h'(z)}{\lvert \ln z \rvert} = c_9 > 0. \label{eq20}
	\end{equation}
	Equality (\ref{eq20}) implies the first statement of (\ref{eq16}) for $k=1$. The second statement of (\ref{eq16}) is proved in the same way: the positive definiteness of $a^{(2)}(x)$ implies that $c_{4,i} > 0$. The proof of (\ref{eq16}) for $k > 1$ is similar. One just should take into account that the set $\mathcal{E}_z^h$ in this case invariant with respect to rotations preserving quadratic forms $u_1^2$ and $u_2^2$, and in formula (\ref{eq19}) $A_{2n-2}$ should be replaced by the product of areas of spheres in $\R^{k-1}$ and in $\R^{l-1}$.
\end{proof}
\begin{proof}[Remark.]
	Let $O_0$ be an interior vertex with $k=1, l=2n-1$ and $I_i \sim O_0$. One can derive that 
$$\lim_{(z,i)\rightarrow 0} \lvert \frac{d}{dz} \left( \frac{1}{\sqrt{\overline{a}_i^{(2)}}(z)} \right) z \cdot (\ln |z|)^{3/2} \rvert = c_{i,10} \in (0,\infty)$$
This implies that the function $$\lvert \overline{a}^{(2)}(z) \cdot \frac{d}{dz} \left( \frac{1}{\sqrt{\overline{a}^{(2)}(z)}} \right) \rvert$$ is of order $\frac{1}{\lvert z (\ln z)^{3/2} \rvert}$ as $z \rightarrow 0$ and is integrable at 0.
\end{proof}
Let $C_{0,T}(\Gamma)$ be the space of continuous functions $\varphi: [0,T] \rightarrow \Gamma$ provided with the uniform topology and let $K$ be a compact subset of $\Gamma$.
\vspace{2.5mm}\\
\noindent \textbf{Lemma 7.} \label{lem7} The measures in $C_{0,T}(\Gamma)$ induced by the process $\Y^{\delta}(t)$, $\Y^{\delta}(0) \in K, \delta \in (0,1)$, are tight in the weak topology in $C_{0,T}(\Gamma)$.
\begin{proof}
	The proof is similar to that of Lemma 8.3.2 from \cite{bib9}, and we thus omit it.
\end{proof}
Let $O_0$ be an interior vertex and $I_1, I_2, I_3 \sim O_0$. Denote by $U_h = U_h(O_0)$ the $h$-neighborhood of $O_0 \in \Gamma$, and let $$\tau_h^{\delta} = \min \{ t : \Y^{\delta}(t) \in \partial U_h \},$$ where $\partial U_h$ is the boundary of $U_h$. As usual, we use the subscript for the probability and expectation to indicate the initial point.

The distribution of $\Y^{\delta}(\tau_h^{\delta})$ and $E_y\tau_h^{\delta}$ as functions of the starting point $y \in U_h(0)$ can be found explicitly as solutions of ordinary differential equations in $U_{\delta}(O_0)$ with the gluing conditions (\ref{eq9}) at $O_0$ and corresponding boundary conditions on $\partial U_h(O)$. This calculation as well as the calculations of the limit as $\delta \rightarrow 0$ are similar to Lemmas 2.2 and 2.3 of \cite{bib2}. Therefore, we omit the proofs of the next two lemmas.
\vspace{2.5mm}\\
\noindent \textbf{Lemma 8.} \label{lem8} Let $O$ be an interior vertex and $I_1, I_2, I_3 \sim O$. Assume that assumptions (A1) - (A8) are satisfied.
	
	If $I_1$ is the only exit edge from $O$, then for a small enough $$\lim_{\delta \rightarrow 0} P_{O_0}\{\Y^{\delta}_{\tau_h^{\delta}} \in I_1 \} = 1.$$
	
	If there are two exit edges $I_1$ and $I_2$, then for a small enough
	\begin{align*}
		\lim_{\delta \rightarrow 0} P_{O_0} \{ \Y^{\delta}_{\tau_h^{\delta}} \in I_3 \} = 0 \\
		\lim_{\delta \rightarrow 0} P_{O_0} \{ \Y^{\delta}_{\tau_h^{\delta}} \in I_j \} = P_j(O_0), j = 1,2,
	\end{align*}
	where $P_1(O_0)$ and $P_2(O_0)$ are given by formula (\ref{eq14}).
\vspace{2.5mm}\\
\noindent \textbf{Lemma 9.} \label{lem9} Let the conditions of Lemma 8 be satisfied. Then there exists $h_0 > 0$ and $A > 0$ such that for any $h \in (0,h_0)$, any $y \in U_h(O_0)$ and any small enough $\delta > 0$: $$E_y\tau_h^{\delta} \leq Ah \lvert \ln h \rvert$$
\vspace{2.5mm}\\
\noindent \textbf{Lemma 10.} \label{lem10} Let points $(z_1,i)$ and $(z_2,i)$ belong to the interior of an edge $I_i \subset \Gamma, z_1 < z_2, I_i(z_1,z_2) = \{(z,i) \in I_i : z_1 < z < z_2\}$. Put $\tau^{\delta}(z_1, z_2) = \min \{t : \Y^{\delta}(t) \notin I_i (z_1, z_2)\}$. Then, for any $h > 0, t > 0$ and $y \in I_i(z_1, z_2)$. $$\lim_{\delta \rightarrow 0} P_y \{ \max_{0 \leq s \leq t \wedge \tau^{\delta}} \lvert\Y^{\delta}(s) - \Y (s)\rvert > h\} = 0$$  
\begin{proof}
	The proof of this last lemma is simple and follows from Theorem 1.1.2 of \cite{bib9}.
\end{proof}

\begin{proof}[Proof of Theorem \ref{thm2}.] Strong Markov property of process $\Y^{\delta}(t)$ and Lemmas 8-10 provide the convergence of finite-dimensional distributions of $\Y^{\delta}(t)$ to corresponding distributions of $\Y(t), \Y(0) = \Y^{\delta}(0) = y$. Together with the tightness of distributions given by Lemma 7, this implies the weak convergence of processes $\Y^{\delta}(t)$ to $\Y(t)$. 
\end{proof}

Note that for each $y \in \Gamma$, process $\Y(t)$, $\Y(0) = y$, has just a finite number of trajectories. Stochasticity of $\Y(t)$ appears just when it goes through an essential vertex -- a vertex of type $\frac12$ or $\frac21$ with two exit edges attached to the vertex.

\noindent \textbf{Assumption (A7):} Let $\hat{b}_i(z)$ has just a finite number of zeros on $\Gamma$, and if $\hat{b}_{i_0}(z_0) = 0$ for an interior point of $I_{i_0}$ then $\hat{b}_{i_0}(z)$ change the sign at this point.

\vspace{2.5mm}

The function $\hat{b}_i(z)$ is equal to zero at exterior vertices. Such a vertex $O_0$ can be stable if $Y^{-1}(O_0)$ is a minimum of $H(x)$ and $\hat{b}_i(z)$ is negative in a vicinity of $O_0$ or if $Y^{-1}(O_0)$ is a maximum of $H(x)$ and $\hat{b}_i(z)$ is positive in a vicinity of $O$. The process $\Y(t)$ can have stable equilibria inside the edges. Let $(z_0, i_0)$ be a stable equilibria inside $I_{i_0}$. Denote by $\mu_{z_0,i_0}$ the measure in $\R^{2n}$ concentrated on $C_{i_0}(z_0)$ with the density $$\lvert \nabla H(x) \oint_{C_{i_0}(z_0)} \frac{dS}{\lvert \nabla H(y) \rvert} \rvert^{-1}, \quad x \in C_{i_0}(z_0).$$ If $(z_0, i_0)$ is an exterior vertex we denote by $\mu_{z_0, i_0}$ the $\delta$-measure at $Y^{-1}(z_0, i_0) \in \R^{2n}$.

We say that the \textbf{Assumption (A8)} is satisfied if for any $(z_1, i) \in \Gamma$ there is $z_2 > z_1$ such that $\hat{b}_i(z_2) < 0$. This assumption implies that each trajectory $\Y(t)$ stay in a bounded part of $\Gamma$ forever.

Let assumptions (A1) - (A8) be satisfied. Then for each initial point $y \in \Gamma$, there exists a finite set $\Phi(y)$ of asymptotically stable points which can be limiting points for trajectories of $\Y(t), \Y(0) = y$, as $t \rightarrow \infty$.

For each $y \in \Gamma$ and each $(z_0, i_0) \in \Phi(y)$, denote by $\Y[y \rightarrow (z_0,i_0)]$ the trajectory leading from $y$ to $(z_0, i_0)$. If such a trajectory exists it is unique. Denote by ess$[y \rightarrow (z_0,i_0)]$ the set of essential vertices belonging to $\Y[y \rightarrow (z_0,i_0)]$.

For each vertex $(z,i) \in \text{ess}[y\rightarrow (z_0, i_0)]$, there are two exit edges attached to $(z,i)$: edge $I_{k_1}(z,i)$ along which $\Y(t)$ goes to $(z_0,i_0)$, and another edge $I_{k_2}(z,i)$.

For each $(z,i) \in \text{ess}[y \rightarrow (z_0, i_0)]$ put
\begin{equation} \label{xx1}
	P[(z,i) \rightarrow (z_0, i_0)] = \int_{G_{k_1}(z)} \text{div} b(x) dx \left( \int_{G_{k_1}(z) \cup G_{k_2}(z)} \text{div} \left( b(x) \right)dx \right)^{-1},
\end{equation}
\begin{equation} \label{xx2}
	P[y \rightarrow (z_0, i_0)] = \prod_{(z,i) \in \text{ess}[y \rightarrow (z_0, i_0)} P[(z,i) \rightarrow (z_0, i_0)]. 
\end{equation}
Let, for instance, the Reeb graph be as shown in Figure 2, $\Y(0) = x_0,$ and $\hat{b}_i(z) < 0$ on the whole $\Gamma$ except the exterior vertices of $\Gamma$ where $\hat{b}_i(z) = 0$. The asymptotically stable equilibria in this case are the exterior vertices corresponding to the minimum points of $H : \Phi(x_0) = \{ O_5, O_6, O_{11}, O_{12} \}$. Essential vertices on the path leading from $x_0$, say, to $O_{11}$ are $O_2$ and $O_{10}$: ess$[x_0 \rightarrow O_{11}] = \{ O_2, O_{10} \}$. To go from $x_0$ to $O_{11}$, trajectory $\Y(t), \Y(0) = x_0$, should go along $I_7$ from $O_2$ and along $I_{11}$ at the vertex $O_{10}$.

It is easy to derive from Theorems \ref{thm1} and \ref{thm2} the following result.
\vspace{2.5mm}\\
\noindent \textbf{Theorem 3.} \label{thm3} Let assumptions (A1) - (A8) be satisfied. Then for any $h>0, y \in \Gamma$, and any continuous function $f(x), x \in \R^{2n}$, there exists $T_0$ such that for $T > T_0$
	\begin{equation} \label{xx3}
		\lim_{\delta \rightarrow 0} \lim_{\ve \rightarrow 0} \lvert E_y f(X^{\ve}_T) - \sum_{(z,i) \in \Phi(y)} \int_{\R^{2n}} f(x) \mu_{(z,i)}(dx) \cdot p[y \rightarrow (z,i)] \rvert < h.
	\end{equation}
\begin{example}
	Let $n=2, x=(p,q),p=(p_1,p_2) \in \R^2, q = (q_1, q_2) \in \R^2, H(x) = \frac{\lvert P \rvert^2}{2} + F(q)$. Let $\Gamma$ be the Reeb graph for $H(x)$. Assume that $F(q), q \in \R^2$, is a Morse function and $\lim_{\lvert q \rvert \rightarrow \infty} F(q) = \infty$. Let $F(q)$ has two local minima at $q^{(1)}$ and at $q^{(2)}$ and one maximum at $q^{(3)}$ ($q^{(1)}, q^{(2)}, q^{(3)} \in \R^2$). Because of topological arguments, such a function should have also two saddle points, say at $q^{(4)}$ and $q^{(5)} \in \R^2$. Then $H(x)$ has mimima at $(0,0,q^{(1)}), (0,0,q^{(2)}) \in \R^4$, no maxima, and 3 saddle points. The results of Section 2 imply that one of the saddle points corresponds 
to a $\frac12$-type vertex (denote this vertex by $O$), another - to a vertex of $\frac21$-type, and yet another - to a $\frac11$-type vertex.
	
	Let $b(x) = (b_1(x), b_2(x), b_3(x), b_4(x))$ where $(b_1(x), b_2(x)) = -Ap$, where $A = (A_{ij})$ is a constant non-degenerate $2 \times 2$-matrix with negative trace: $\text{tr} A = A_{11} + A_{22} = -\lambda <0$.
	
	To describe the process $\Y(t)$ in an explicit way, introduce new coordinates on $\Gamma:(z,i) - (v,i)$, where the second coordinate $i$ is the same, and $v = v_i(z)$ is the volume of domain $G_i(z)$ bounded by $C_i(z)$. Since div $b(x) = -2\lambda$, 
	$\hat{b}_i(z) = -2\lambda v(z)$. The evolution of $\Y(t)$ inside an edge $I_i$ is described by equations (\ref{eq13}). In our case these equations have the form
	\begin{equation} \label{eq24}
		v_i'(z_t)\dot{z}_t = -2\lambda v_i(z_t)
	\end{equation}
	Let $\tilde{v}_i(t) = v_i(z_t)$. Then (\ref{eq24}) implies that $\tilde{v}_i(s+t) = \tilde{v}_i(s)e^{-2\lambda t}$ until $\Y(s)$ and $\Y(t)$ belong to $I_i$. When $\Y(t)$ hits a vertex of $\frac11$-type or $\frac21$-type, it goes without any delay to the unique exit edge attached to this vertex. If $\Y(t)$ hits vertex $O$, which has two exit edges: an edge $I_{i_1}$ leading to the minimum $(0,0,q^{(1)})$ of $H(x)$ and an edge $I_{i_2}$ leading to $(0,0,q^{(2)})$, then $\Y(t)$ without delay goes to one of these edges with probabilities $$p_{i_1}(0) = \frac{v_1}{v_1 + v_2}, p_{i_2}(0) = \frac{v_2}{v_1 + v_2}.$$ Here $v_1$ is the volume of one of two domains attached to $Y^{-1}(0)$ and bounded by $C(H(0)) = \{ x \in \R^4 : H(x) < H(0) \}$, namely the domain containing $Y^{-1}(I_{i_1})$, $v_2$ is the volume of another domain attached to $Y^{-1}(0)$.
	
	If trajectory $\Y(t), \Y(0) = y$, comes to $O$, then $\Phi(y)$ consists of two points. In this case, for large $t$ and $0 < \ve \ll \delta \ll 1, X^{\ve,\vk,\delta}(t)$ is distributed between small neighborhoods of points $(0,0,q^{(1)}), (0,0,q^{(2)}) \in \R^4$ with probabilities $p_i = \frac{v_i}{v_1 + v_2}, i \in \{1,2\}$, If $\Y(t), \Y(0) = y$, does not go through $O$, $\Phi(y)$ consists just of one exterior vertex, and $X^{\ve,\vk,\delta}(t)$ will be close to the corresponding minimum point of $H(x)$. 
\end{example}

\section{Remarks and Generalizations}
\textbf{1.} Theorems \ref{thm1} and \ref{thm2} describe the double limit of $\Y(X^{\ve,\delta}_t)$ on any finite, independent of $\ve$ and $\delta$, time interval $[0,T]$. Similar results hold in the case $T = T(\delta)$ growing as $\delta \rightarrow 0$ but slow enough: if $\delta \ln T(\delta) \rightarrow 0$ as $\delta \rightarrow 0$. If we are interested in longer time intervals, for instance, in deviations of order 1 as $\delta \rightarrow 0$ from an asymptotically stable regime or in transitions between such stable regimes, one should take into account large deviations of the process $\Y^{\delta}(t)$ from $\Y(t)$. Those deviations are governed by the action functional $\frac{1}{\delta}S_{0,T}(\varphi) = \frac{1}{2\delta} \int_0^T \lvert (\overline{a^{(2)}}_i(\varphi_s))^{-1/2} (\dot{\varphi}_s - \overline{b}_i(\varphi_s)) \rvert^2 ds$ (see \cite{bib9}). Using this action functional, one can describe various asymptotics in the exit problem from a neighborhood of an asymptotically stable regime as well as the metastability effects for different time scales (compare with \cite{bib9}).

\vspace{2.5mm}
\noindent \textbf{2.} To be specific, we considered deterministic perturbations of Hamiltonian system by a deterministic vector field $b(x)$ combined with a qualified noise. But, actually, our approach works in a more general situation, for instance, in the case of a divergence free dynamical system having a smooth first integral. Consider, for example, a system in $\R^3$
\begin{equation} \label{xx5}
	\dot{X}(t) = \nabla F(X(t)) \times \ell (X(t)), \quad X(0) = x \in \R^3.
\end{equation}
The function $F(x)$ and the vector field $\ell(x)$ assumed to be smooth enough. It is easy to see that $F(x)$ is a first integral for system
(\ref{xx5}). Let $\lim_{\lvert x \rvert \rightarrow \infty} F(x) = \infty$ and assume also that curl $\ell(x) \equiv 0$. Then, 
\[
	\text{div} \left(\nabla F(x) \times \ell(x)\right) = \ell(x) \cdot \text{curl} \nabla F(x) - \nabla F(x) \cdot \text{curl} \ell(x) = 0, 
\]
 so that the vector field $\nabla F(x) \times \ell(x)$ is divergence free. Thus, the Lebesgue measure in $\R^3$ is invariant for system 
(\ref{xx5}). In this case, the long-time evolution of the system obtained from (\ref{xx3}) after addition to right-hand side a deterministic perturbation combined with a qualified noise can be described in the same way as in the case of Hamiltonian system. Using the results of \cite{bib8} 
instead of \cite{bib7}, perturbation of more general systems can be considered.

\vspace{2.5mm}
\noindent \textbf{3.} Up to now, we considered perturbations of deterministic systems. Sometimes, perturbations of stochastic systems are of interest. Let $X_t$ be the Markov process in $\R^d$ governed by the operator $L$, $Lu(x) = \frac12 \text{div} (a_1(x) \nabla u(x))$, where $a(x)$ is a non-negative definite $d \times d$-matrix.

Assume a Morse function $H(x)$, $\lim_{\lvert x \rvert \rightarrow \infty} = \infty$, exists such that $a_1(x) \nabla H(x) = 0$ for $x \in \R^d$. This condition implies that $H(x)$ is a first integral for the process $X_t$ (see, for instance, \cite{bib7}). Let the process $X_t$ be perturbed by a small vector field combined with even smaller noise. More precisely, let after an appropriate time change, the perturbed process $\dot{X}_t^{\ve, \delta}$ is governed by the operator $L^{\ve,\delta}$:
\begin{align*}
	L^{\ve,\delta}u(x) = \frac{1}{\ve}Lu + b(x) \cdot \nabla u + \frac{\delta}{2}\text{div}(a_2(x)\nabla u(x)),
\end{align*}
where $b(x), x \in \R^d$, is a smooth vector field and $a_2(x)$ is positive definite $d \times d$-matrix.

Let $\Gamma$ be the Reeb graph for $H(x)$, and $Y : \R^d \rightarrow \Gamma$ is the projection of $\R^d$ on $\Gamma$. Then one can expect that the process $\Y_t^{\ve,\delta} = Y(X_t^{\ve,\delta})$ converge weakly as first $\ve \rightarrow 0$ and then $\delta \rightarrow 0$ to a Markov process $\Y_t$ on $\Gamma$. The characteristics of $\Y_t$ can be calculated similarly to the case considered in Sections 3 and 4.

\vspace{2.5mm}
\noindent \textbf{4.} The results of this paper can be reformulated in the terms of partial differential equations, and our approach can be useful in some PDE problems. Consider, for example, the following quasi-linear Cauchy problem:
\begin{align*}
	\frac{\partial u^{\ve,\delta}(t,x)}{\partial t} = \frac{1}{2\ve}\text{div}(a_1(x)\nabla u^{\ve,\delta}) + b(x,u^{\ve,\delta}) \cdot \nabla u^{\ve,\delta} + \frac{\delta}{2}\text{div}(a_2(x)\nabla u), \\
	u^{\ve,\delta}(0,x) = g(x),
\end{align*}
where $a_1(x)$ a non-negative definite and $a_2(x)$ is a positive definite matrices. Assume that a Morse function $H(x)$, $\lim_{\lvert x \rvert \rightarrow \infty}H(x) = \infty$, exists such that $a_1(x) \nabla H(x) \equiv 0$. Then, one can expect that when $\ve \rightarrow 0$ we get a Cauchy problem for a quasi-linear equation with the viscosity of order $\delta$ and gluing conditions independent of $u$. Then, taking $\delta \rightarrow 0$ we get a first order quasi-linear equation with nonlinear gluing conditions. The solutions of this equation usually have discontinuities -- shock waves. One can expect that in this way, we can get some information about shock waves in the multi-dimensional equation (\ref{eq24}).

\vspace{2.5mm}
\noindent \textbf{5.} Systems with several first integrals can be considered. In this case, the limiting process lives on an open book space defined by the first integrals (see \cite{bib9}, Ch. 8 and 9). Then the limiting process will be a deterministic motion inside the pages with gluing conditions on the binding of the open book.

\section*{Acknowledgements}
Finally, I would like to thank Alexander Wentzell for useful discussions of the problems considered in this paper.


\begin{thebibliography}{9}

\bibitem{bib1} Adelson-Velskii A.G., Kronrod A.S. (1945) About level sets of continuous functions with partial derivatives,Doklady Akad. Nauk SSSR, 49, 4,pp. 239-241.

\bibitem{bib2} Brin M.I., Freidlin M.I., (2000) On stochastic behavior of perturbed Hamiltonian Systems, Ergod. Th. and Dynam. Systems, 20, pp.55-76.

\bibitem{bib3} Dolgopyat, D., Freidlin M., Koralov L., (2012) Deterministic and stochastic perturbations of area-preserving flows on two dimensional torus, Ergod. Th. and Dynam. Systems, 32, pp.899-918.

\bibitem{bib4} Freidlin M., (2022) Long-time influence of small perturbations and motion on simplex of invariant probability measures, Pure and Appl. Functional Analysis, 7, pp.551-592.

\bibitem{bib5} Freidlin  M., Hu W., (2011) On perturbations of generalized Landau-Lifshitz equations, Journ.Stat.Phys., 144, pp.978-1008.

\bibitem{bib6} Freidlin, M., Sheu Shuenn-Jyi, (2000) Diffusion processes on graphs: stochastic differential equations, large deviations, Probab. Theory and  Related Fields, 116, pp. 181-220.

\bibitem{bib7} Freidlin M., Weber M., (2001) On random perturbations of Hamiltonian systems with many degrees of freedom, Stochastic Processes and Appl., 94,
pp. 199-139.

\bibitem{bib8} Freidlin M., Weber M., (2004) Random perturbations of dynamical systems and diffusion processes, Probab. Th. and  Related Fields, 128, pp.441-466.

\bibitem{bib9} Freidlin M., Wentzell A., (2012) Random Perturbations of Dynamical Systems, Springer, Third edition.

\bibitem{bib10} Gorban A., (2013) Thermodynamic tree:The space of admissible paths,SIAM Journal Appl. Dynamical Systems, 12, pp.246-278.

\bibitem{bib11} Khasminskii R., (1963) Principle of averaging for parabolic and elliptic differential equations and Markov processes with small diffusion, Theory Probab. Appl.8, pp. 1-21.

\bibitem{bib12} Reeb G. (1946) Sur les points singuliers d'une forme de Pfaff completement integrable ou d'une fonction numérique, C. R. Acad. Sci. Paris, 222, pp.847-849.

\end{thebibliography}
\end{document}